\documentclass[11pt, leqno]{amsart}
\usepackage{amsthm,amsmath,amssymb,amscd,verbatim}
\usepackage[curve,matrix,arrow,frame,tips]{xy}
\DeclareMathOperator{\GL}{GL}

\newcommand{\mA}{{\mathbb A}}

\newcommand{\mF}{{\mathbb F}}
\newcommand{\mG}{{\mathbb G}}

\newcommand{\mP}{{\mathbb P}}
\newcommand{\mQ}{{\mathbb Q}}
\newcommand{\mR}{{\mathbb R}}

\newcommand{\mZ}{{\mathbb Z}}

\newcommand{\F}{{\mathcal F}}

\newcommand{\N}{{\mathcal N}}

\renewcommand{\to}{\longrightarrow}

\theoremstyle{plain}
\newtheorem{thm}{Theorem}[section]
\newtheorem{lemma}[thm]{Lemma}
\newtheorem{coro}[thm]{Corollary}
\newtheorem{prop}[thm]{Proposition}

\theoremstyle{definition}

\newtheorem{rk}[thm]{Remark}

\numberwithin{equation}{section}

\begin{document}
\title[Deligne-Lusztig varieties and period domains]{Deligne-Lusztig varieties and period domains over finite fields}
\author[S. Orlik]{S. Orlik}
\address{ Universit\"at Leipzig\\
Fakult\"at f\"ur Mathematik und Informatik\\
Postfach 100920\\
 04009 Leipzig\\ Germany.}
\email{orlik@mathematik.uni-leipzig.de}
\author[M. Rapoport]{M. Rapoport}
\address{Mathematisches Institut der Universit\"at Bonn,  
Beringstrasse 1\\ 53115 Bonn\\ Germany.}
\email{rapoport@math.uni-bonn.de}

\date{\today}
 
\maketitle


\section{Introduction}
Let $G_0$  be a reductive group  over $\mF_q$. There are two classes of algebraic varieties over an algebraic closure
$\mF$ of $\mF_q$ attached to $G_0.$ Let us recall their definition. 
We set $G=G_0\times_{\mF_q} \mF.$

To $G$ there is associated {\it the} maximal torus, {\it the} Weyl  group $W$ and {\it the} set of fundamental reflections in $W$, 
cf. \cite{DL} 1.1.
Let $X=X_G$ be the set of all Borel subgroups of $G$. Then $X$ is  a smooth projective algebraic variety homogeneous
under $G$. The set of orbits of $G$ on $X\times X$ can be identified with $W,$ and this defines {\it the relative
position map} ${\rm inv}:  X \times X \rightarrow W$ (associate to an element of $X\times X$ the $G$-orbit containing
it). Let $w \in W.$ The {\it Deligne-Lusztig variety associated to} $(G_0,q,w)$ is the locally closed subset of $X$ given by
\begin{equation*} 
X(w)=X_{G_0}(w)=\{x\in X\mid {\rm inv}(x,Fx)=w\}\ .
\end{equation*}
 Here $F: X \rightarrow X$ denotes the Frobenius map over $\mF_q$. 
It is known (\cite{DL}, 1.4) that $X(w)$ is a smooth quasi-projective variety of 
dimension $\ell(w)$, which is equipped with an action of $G_0(\mF_q)$. If $F^e$ is the minimal 
power of $F$ with $F^e(w)=w$, then $X(w)$ is defined over $\mF_{q^e}.$ 

For the other class of varieties, fix a conjugacy class $\N$ of cocharacters $\nu:\mG_m \rightarrow G$. Any such
$\nu$ defines a parabolic subgroup $P_\nu$ of $G$ and all parabolics obtained from elements $\nu\in \N$ are conjugate.
Let $X_G(\N)$ be the set of these conjugates, a smooth projective algebraic variety homogeneous under $G$. Any $\nu\in \N$
defines via the adjoint representation a $\mZ$-filtration $\F_\nu$ on ${\rm Lie}(G)$, and $\nu$ is called {\it semi-stable}
if $({\rm Lie}(G_0), \F_\nu)$ is semi-stable as a $\mF_q$-vector space equipped with a $\mZ$-filtration on the corresponding
$\mF$-vector space, cf.  \cite{R, F}. This condition only depends on the point in $X(\N)$ corresponding to $\nu$
and defines an open subset $X(\N)^{\rm ss}=X_{G_0}(\N)^{\rm ss}$ of $X(\N)$, called the {\it  period domain associated with} $(G_0,q,\N)$, cf. \cite{R}.
Hence $X(\N)^{\rm ss}$ is a smooth connected quasi-projective variety of dimension
$\dim X(\N)$. It is equipped with an action of $G_0(\mF_q)$. If the conjugacy class $\N$ is defined over the extension  $\mF_{q^e}$, then $X(\N)^{\rm ss}$
is defined over $\mF_{q^e}$. 

The Drinfeld space $\Omega^n$ (relative to $\mF_q$) is a DL-variety, as well as a period domain. More precisely, let $G_0=\GL_n.$
Let $w=s_1s_2\cdots s_{n-1}=(1,2,\ldots,n)$ be the standard Coxeter element. Then $X_{G_0}(w)$ can be identified with the Drinfeld space
$$\Omega^n=\Omega^n_{\mF_q}=\mP^{n-1} \setminus \bigcup\nolimits_{H / \mF_q} H$$
(complement of all $\mF_q$-rational hyperplanes in the projective space of lines in $\mF^n$), cf. \cite{DL}, \S 2.
For any Coxeter element $w$ for $\GL_n$, the corresponding DL-variety $X(w)$ is universally homeomorphic to $\Omega^n$, cf. \cite{L'}, Prop. 1.10.
On the other hand, let us identify as usual the set of conjugacy classes $\N$ for $\GL_n$ with
$$(\mZ^n)_+=\big\{(x_1,\ldots,x_n)\in \mZ^n \mid x_1\geq x_2 \geq \ldots \geq x_n \big\}.$$
Let $(x,y^{(n-1)})\in (\mZ^n)_+$ with $x > y$ (here $y^{(n-1)}$ indicates that the entry $y$ is repeated $n-1$ times).
Then the corresponding period domain is equal to $\Omega^n$, cf. \cite{R}. Similarly, if   $(x^{(n-1)},y)\in (\mZ^n)_+$ with $x>y$, then
the corresponding period domain is isomorphic to $\Omega^n$ (it is equal to the dual $\check{\Omega}^n$, the set of hyperplanes of $\mF^n$ not containing any $\mF_q$-rational line).

In \cite{R}, \S 3, it is shown on examples that the Drinfeld space has various special features that are not shared
by more general period domains. In the present paper we exhibit another such feature: the Drinfeld space is essentially
the only period domain which is at the same time a DL-variety. Before
formulating the result, we note that both $X_{G_0}(w)$ and $X_{G_0}(\N)^{\rm ss}$ only depend on the adjoint group $G_{0\; \rm ad}$. Also,
if $G_0$ is the direct product of  groups , then the corresponding Deligne-Lusztig varieties and period domains also split
into products. Hence we may assume that $G_0$ is $\mF_q$-simple and adjoint. Then $G_0$ is of the form $G_0={\rm Res}_{\mF_{q'}/ \mF_q}(G_0')$ for an absolutely simple group $G_0'$ over $\mF_{q'}.$ Then $\N$ is of the form $(\N_1, \ldots , \N_t)$ corresponding to the $\mF_q$-embeddings of $\mF_{q'}$ into $\mF$. Here $t=|\mF_{q'}: \mF_q|$ and $\N_1, \ldots, \N_t$ are conjugacy classes of $G'$. 

The main result of this note is the following theorem.

\begin{thm}\label{Theorem}
Let $G_0$ be absolutely simple of adjoint type over $\mF_q.$ A Deligne-Lusztig variety $X_{G_0}(w)$ is never universally homeomorphic to a period domain $X_{G_0}(\N)^{\rm ss}$, unless $G_0={\rm PGL}_n$, $w$ is
a Coxeter element and $\N$ corresponds to $\nu \in (\mZ^n)_+$ of the form $\nu=(x,y^{(n-1)})$ or $\nu=(x^{(n-1)},y)$
with $x>y$, in which case $X_{G_0}(w)$ and $X_{G_0}(\N)^{\rm ss}$ are both universally homeomorphic to $\Omega^n_{\mF_q}.$

More generally, let $G_0={\rm Res}_{\mF_{q'}/ \mF_q}(G_0')$ be simple of adjoint type, and let  $t=|\mF_{q'}: \mF_q|$. Then a Deligne-Lusztig variety $X_{G_0}(w)$ is never universally homeomorphic to a period domain $X_{G_0}(\N)^{\rm ss}$, unless $G_0'={\rm PGL}_n$, $w$ is
a Coxeter element in the sense of \cite{L'}, 1.7,  and $\N$ is of the form $(\nu_1, \ldots , \nu_t)\in ((\mZ^n)_+)^t$ with $\nu_i$  scalar for all indices $i=1, \ldots , t$, except one index where the entry  is of the form $(x,y^{(n-1)})$ or $(x^{(n-1)},y)$
with $x>y$. In this case $X_{G_0}(w)$ and $X_{G_0}(\N)^{\rm ss}$ are both universally homeomorphic to $\Omega^n_{\mF_{q'}}.$

\end{thm}

 This theorem comes about by comparing a cohomology vanishing theorem for the DL-varieties with a cohomology non-vanishing theorem for
 period domains. 
 In the sequel we denote for any variety $X$ over $\mF$ by $H^i_c(X)$ the $\ell$-adic cohomology group with compact supports $H^i_c(X,\overline{\mQ}_\ell).$ 
 
The vanishing theorem for DL-varieties is the following statement.
\begin{prop}\label{PropositionI}
$$H_c^i(X_{G_0}(w))=0 \mbox{ for } 0\leq i<l(w).$$
\end{prop}

 This vanishing property is due to  Digne, Michel and Rouquier  \cite{DMR}, Cor. 3.3.22. When 
$q\geq h$ (where $h$ denotes the Coxeter number of $G$) then all DL-varieties $X_{G_0}(w)$ are affine, cf. \cite{DL}, Thm. 9.7. In this case, the vanishing statement follows by Poincar\'e duality from a general vanishing theorem
for the \'etale cohomology of affine varieties. Before we became aware of  the paper 
\cite{DMR}, we pursued a strategy for proving Proposition \ref{PropositionI}, which relates its statement to the general problem of determining which DL-varieties are affine. Since we believe that our approach has its own merits,
we give it in \S 2. It seems more elementary than the approach in \cite{DMR}, and is also applicable to
the Deligne-Lusztig local systems on DL-varieties. However, we did not succeed completely, since we have to base ourselves on the following hypothesis.

\medskip
\noindent ${\bf Aff}(G_0,q,w)$: {\it For every $w'$ of minimal length in the $F$-conjugacy class of $w$, the corresponding DL-variety $X_{G_0}(w')$ is affine.}

\medskip  
It seems to us quite likely that this condition
is always satisfied.\footnote{ X.~He \cite{H} has recently given a proof of this conjecture which is inspired by our method in \S 5.  } Lusztig's result \cite{L'}, Cor. 2.8,  that  $X_{G_0}(w)$ is affine when $w$ is a Coxeter element may be viewed as supporting this belief. In any case,  we show that ${\bf Aff}(G_0,q,w)$ is satisfied when 
 $G_0$ is a split classical group (cf. \S 5). It is also satisfied when $G_0$ is of type $G_2$, cf. \cite{H2}, 4.18. On the other hand, we believe that the hypothesis that $w$ be of minimal length in its conjugacy class cannot be totally dropped, i.e., we believe it may happen for small $q$ that there are DL-varieties which are not affine, although we have no example for this (but a concrete candidate over the field with $2$ elements, cf. Remark \ref{Haastertcomment}).

On the other hand, there is the following non-vanishing result \cite{O}, Cor. 1.2 for period domains. Let $r_0={\rm rk}_{\mF_q} (G_0)$ denote the $\mF_q$-rank of $G_0$ (dimension of a maximal $\mF_q$-split torus of $G_0$).
\begin{prop}\label{Non-vanishing}
 Let $G_0$ be a simple group of adjoint type over $\mF_q$. If $\N$ is non-trivial, then 
\begin{equation*}
H_c^{r_0}(X_{G_0}(\N)^{\rm ss})\neq 0 \; ;
\end{equation*}
in fact, the representation of $G_0(\mF_q)$ on this cohomology group is irreducible and is equivalent to the Steinberg representation.
\end{prop}

In order to carry out the comparison between these two results, we use the following observation.

\begin{prop}\label{PropositionII} Let $G$ be a simple group of adjoint type over an algebraically closed field $k$. For any proper
parabolic subgroup $P$, the following inequality holds,
$${\rm rk} (G) \leq \dim G/P ,$$
with strict inequality, except when $G={\rm PGL}_n$ and $P$ is a parabolic subgroup of type $(n-1,1)$ or $(1,n-1).$
\end{prop}

Our approach to  Proposition \ref{PropositionI} is given in \S 2, and the proof  of Proposition \ref{PropositionII} in \S 3. The main theorem is  proved in \S 4. In the final section \S 5, we verify the condition ${\bf Aff}(G_0,q,w)$
for split classical groups by checking the criterion of Deligne and Lusztig \cite{DL}, 9.6.
 
\smallskip

{\it Acknowledgements:} We thank L. Illusie and Th. Zink for helpful discussions on $\ell$-adic cohomology.

\section{A vanishing theorem }

Let $\mathcal F$ be a smooth $\bar \mQ_\ell$-sheaf on a connected {\it normal} variety $X$ over $\mF$. We say that $\mathcal F$ is a  {\it smooth prime-to-$p$ $\bar \mQ_\ell$-sheaf}, if it is defined by a constant tordu sheaf and the corresponding representation of the fundamental group $\pi_1(X)=\pi_1(X, x)$
on the fiber $\mathcal F_x$ at a geometric point $x$ of $X$   factors through the prime-to-$p$ part
$\pi_1(X)^{(p)}$ of $\pi_1(X)$.  This is independent of the choice of $x$. The extension of this definition to non-connected normal schemes is immediate. 

We will use the following stability property of smooth prime-to-$p$   $\bar \mQ_\ell$-sheaves. Let $S$ be a normal scheme and let $f: X\to S$ be a smooth morphism of relative dimension one, with all fibers
affine curves. We assume that $f$ factors as $f=\bar f\circ j$, where $j:X\hookrightarrow \bar X$ is an open immersion, and where $\bar f: \bar X\to S$ is proper and smooth, and such that $D=\bar X\setminus X$ is a smooth relative divisor over $S$. Let $\mathcal F$ 
be a prime-to-$p$ smooth  $\bar \mQ_\ell$-sheaf on $X$. Then $R^if_!(\mathcal F)$ is  a 
 smooth prime-to-$p$   $\bar \mQ_\ell$-sheaf on $S$ and is trivial for $i\neq 1,2$. Indeed,
$\mathcal F$ is tamely ramified along $D$, so that the smoothness of $R^if_!(\mathcal F)$ follows from \cite{SGA4'}, app. to Th. finitude, esp. 1.3.3 and 2.7.  Also, the vanishing of $R^if_!(\mathcal F)$ for $i\neq 1,2$ follows from the proper base change theorem, and the calculation of the cohomology
of affine curves. Alternatively, one may use Poincar\'e duality to reduce the question to the analogous statement concerning $R^if_\ast(\mathcal F)$ (for $i\neq 0,1$ and for the dual sheaf), and refer to \cite{SGA1}, XIII, Prop. 1.14 and Remark 1.17 for the smoothness of $R^if_\ast(\mathcal F)$,  and to loc.~cit., Thm. 2.4, 1) for the commutation of $R^if_\ast(\mathcal F)$ with base change. 
For $i=0,1$, $(R^if_\ast(\mathcal F))_s=H^i(X_s, \mathcal F)$ is equal  to the Galois cohomology group $H^{i}(\pi_1(X_s, x), \mathcal F_x)$, cf. \cite {M}, Thm. 14.14. Under this identification, the action of $\pi_1(S,s)$  is obtained from the action of $\pi_1(X, x)$ on $\mathcal F_x$  in the sense of \cite{S}, I.2.6, b)

[Illusie pointed out to us that this requires  justification. For this, it suffices to prove the analogous statement for a smooth torsion sheaf $\mathcal F$. By restricting $f$ to smaller and smaller open subsets of $S$, we may pass to the generic fiber and are then in the following situation. Let $X$ be an affine smooth curve over a field $k$ and 
let  $\mathcal F$ be a smooth torsion sheaf on $X$. Consider the exact sequence of fundamental groups
$$
1\to \pi_1(X_{\bar k}, x)\to \pi_1(X, x)\to {\rm Gal}(\bar k/k)\to 1 \ .
$$  
The \'etale cohomology groups $H^i(X_{\bar k}, \mathcal F)$ may be identified with the Galois cohomology groups $H^i(\pi_1(X_{\bar k}, x), \mathcal F_x)$ since the inverse image of $\mathcal F$ to the universal covering of $ X_{\bar k}$ is acyclic \cite {M}, Thm. 14.14. There are two actions of ${\rm Gal}(\bar k/k)$ on these cohomology groups: one on the Galois cohomology group coming from the fact that the action of $\pi_1(X_{\bar k}, x)$ on $\mathcal F_x$ comes by restricting the action of the bigger group $\pi_1(X,x)$ on $\mathcal F_x$,  and the  action of ${\rm Gal}(\bar k/k)$ on the \'etale cohomology group 
$H^i(X_{\bar k}, \mathcal F)$  by functoriality. It is obvious that these two actions coincide for $i=0$. Since the two functors arise as derived functors, the two actions coincide then for all $i$.]

Now the homomorphism $\pi_1(X,x)\to \pi_1(S,s)$ is surjective \cite{SGA1}, IX, 5.6, hence this action factors through 
$\pi_1(S, s)^{(p)}$. 

\smallskip
 
After these preliminaries, we may state the vanishing theorem.

\begin{thm}\label{vanishing}
Assume ${\bf Aff}(G_0, q, w)$. Let $\mathcal F$ be a smooth prime-to-$p$  $\bar \mQ_\ell$-sheaf on $X(w)$. Then 
$$
H^i_c(X(w), \mathcal F)=0 \ \text { for } 0\leq i<\ell(w) \ .
$$
\end{thm}

For the constant sheaf $\mathcal F=\bar \mQ_\ell$, we obtain the statement of Proposition \ref{PropositionI}, except that here we have to make the hypothesis 
${\bf Aff}(G_0, q, w)$. 

Let $T_0$ be a maximal torus in $G_0$, with corresponding maximal torus $T$ of $G$.  We identify {\it the} Weyl group with the Weyl group of $T$. 
Then to every $w\in W$ and every character 
$\theta: T(\mF)^{wF}\to \bar \mQ_\ell^\times$,  Deligne and Lusztig have associated a smooth prime-to-$p$ sheaf $\mathcal F_\theta$ on $X(w)$, cf. \cite{DL},  p.111 (when $\theta$ is trivial, then $\mathcal F_\theta=\bar \mQ_\ell$).  As an application of Theorem \ref{vanishing} we have the following result.
\begin{coro}
Assume ${\bf Aff}(G_0, q, w)$. For any $\theta$ 
$$
H^i_c(X(w), \mathcal F_\theta)=0 \ \text { for } 0\leq i<\ell(w) \ .
$$
If $\theta$ is nonsingular, then 
$$
H^i_c(X(w), \mathcal F_\theta)=0 \ \text { for } i\neq \ell(w) \ .
$$
\end{coro}
\begin{proof}
The first statement follows from Theorem \ref{vanishing}. For the second statement, we use the fact \cite{DL}, Thm. 9.8 that if $\theta$ is nonsingular, then the natural homomorphism from $H^i_c(X(w), \mathcal F_\theta)$ to $H^i(X(w), \mathcal F_\theta)$ is an isomorphism. Therefore the assertion follows from Poincar\'e duality. 
\end{proof}
\begin{rk}
The previous statement for nonsingular $\theta$ is due to Haastert \cite{H1}, Satz 3.2,  as an application of his result that  $X(w)$ is quasi-affine, cf. \cite{H1}, Satz 2.3. He does not have to assume the hypothesis  ${\bf Aff}(G_0, q, w)$. Of course, when $X(w)$ is affine, this statement is proved in \cite{DL}. 
\end{rk}
 
For the proof of Theorem \ref{vanishing} we  first recall the following result of Geck, Kim and Pfeiffer. Denote by $S$ the set of simple reflections in $W.$ Let $w,w' \in W$ and $s\in S$. Set
$w \stackrel{s}{\rightarrow}_F w'$ if $w'=swF(s)$ and $\ell(w') \leq \ell(w)$. We write
$w \rightarrow_F w'$ if $w=w'$ or if there exist elements $s_1,\ldots,s_r\in S$ and $w=w_1,\ldots, w_r=w'\in W$  with
$w_i \stackrel{s_i}{\rightarrow}_F w_{i+1}, \,i=1,\ldots,r-1.$

\begin{thm}\label{GKP}{\rm ((\cite{GKP}, Thm. 2.6)}
Let $C$ be an $F$-conjugacy class of $W$ and let $C_{min}$ be the set of elements in $C$ of minimal length.
For any $w\in C$, there exists some $w' \in C_{min}$ such that $w\rightarrow_F w'.$
 \end{thm}

We also recall the following lemma (Case 1 of Thm. 1.6 in \cite{DL}). 
 
\begin{lemma}
 Let $w$ and $w'$ be $F$-conjugate. Suppose that there are $w_1,w_2 \in W$ with $w_1w_2 = w$,  $w_2 F(w_1)= w'$
and $\ell(w)=\ell(w_1) + \ell(w_2)= \ell(w').$
Then $X(w)$ and $X(w')$ are  universally homeomorphic and hence $H^\ast_c(X(w), \mathcal F) \cong H^\ast_c(X(w'), \mathcal F)$ for any $\bar \mQ_\ell$-sheaf $\mathcal F$.   
\end{lemma}

As is well-known, this lemma has the following consequence.

\begin{lemma}\label{Lemma_gleicheLaenge}
Let $s\in S$ and let $w,w'\in W$ with $w'=swF(s)$. Suppose that $\ell(w)=\ell(w')$. 
Then $X(w)$ and $X(w')$ are  universally homeomorphic and hence $H^\ast_c(X(w), \mathcal F) \cong H^\ast_c(X(w'), \mathcal F)$ for any $\bar \mQ_\ell$-sheaf $\mathcal F$.    
\end{lemma}

\begin{proof} We consider the following three cases.

\noindent {\bf Case 1:}  $\ell(sw)=\ell(w)-1$. Then we put $w_1=s$, $w_2=sw$ and apply the previous lemma.

\noindent {\bf Case 2:}  $\ell(wF(s))=\ell(w)-1.$  We put $w_1=s$, $w_2=sw'.$ Again we apply the previous lemma, with the roles of $w$ and $w'$ interchanged. 

\noindent {\bf Case 3:}  $\ell(sw)=\ell(w)+1$ and $\ell(wF(s))=\ell(w)+1.$ Then we  apply Lemma 1.6.4 of \cite{DL} to deduce that $w=w'.$
\end{proof}

\noindent {\it Proof of Theorem \ref{vanishing}:}
We prove the claim by  induction on $\ell(w)$.
The case $\ell(w)=0$ is trivial. 

Let $w\in W$ and suppose that the vanishing property holds for all elements in $ W$ with length smaller than $\ell(w)$.
If $w$ is minimal within its $F$-conjugacy class, then the vanishing follows by our assumption ${\bf Aff}(G_0,q,w)$ from Poincar\'e duality and a general vanishing property of affine schemes. 
If $w$ is not minimal,  there is by  Theorem \ref{GKP}   a chain  of simple reflections $s_1,\ldots,s_r\in S$  and $w=w_1,\ldots, w_r\in W$  with
$w_i \stackrel{s_i}{\rightarrow}_F w_{i+1}, \,i=1,\ldots,r-1$ such that $w_r$ is minimal. 
By Lemma \ref{Lemma_gleicheLaenge} and by induction, we may assume that $w'=swF(s)$ where $s\in S$ and $\ell(swF(s))< \ell(w)$, i.e., $\ell(w')= \ell(w)-2$. 
As in the proof of Theorem 1.6 in \cite{DL}, we may write $X(w)$ as a (set-theoretical) disjoint union  
$$X(w)=X_1 \cup X_2$$ 
where $X_1$ is closed in $X(w)$ and $X_2$ is its open complement. By applying the  long exact cohomology sequence,
it suffices to show that $H_c^i(X_1, \mathcal F_{|X_1})=0$ and $H_c^i(X_2, \mathcal F_{|X_2})=0$ for $i<\ell(w).$ Note that the restrictions $\mathcal F_{|X_1}$ and 
$\mathcal F_{|X_1}$ are also prime-to-$p$, since the corresponding representations of their fundamental groups are induced by the canonical maps $\pi_1(X_i)\to \pi(X(w)), \ i=1,2$. Now $X_1$ has the structure of an $\mA^1$-fibering over $X(w')$. Let $f: X_1\to
X(w')$ be the $\mA^1$-fibering. Consider the Leray spectral sequence
\begin{equation*}
H^i_c(X(w'), R^jf_!\mathcal F_{|X_1})\Rightarrow H^{i+j}_c(X_1, \mathcal F_{|X_1}) \ .
\end{equation*}
The stalks of $R^jf_!\mathcal F$ are isomorphic to $H^j_c(\mA^1, \mathcal F_{|\mA^1})$. Now $\pi_1(\mA^1)^{(p)}=0$, cf.  \cite{SGA1}, XIII, Cor. 2.12.  Since $\mathcal F$ is prime-to-$p$, $\mathcal F_{|\mA^1}$ is constant and $H^1_c(\mA^1, \mathcal F)=0$. We deduce that 
\begin{equation*}
H^i_c(X_1, \mathcal F)=H^{i-2}_c(X(w'), R^2f_!\mathcal F) \ .
\end{equation*}
Since $\mathcal F'=R^2f_!\mathcal F$ is a smooth prime-to-$p$ $\bar \mQ_\ell$-sheaf on $X(w')$, the induction hypothesis applies to it and it follows that $H^{i-2}(X(w'), \mathcal F')=0$ for all $i-2<\ell(w')$. 
Thus  $H^{i}_c(X_1, \mathcal F_{|X_1})=0$ for all $i<\ell(w).$

The vanishing of  $H^{i}_c(X_2, \mathcal F_{|X_2})$ is even easier. In the proof of \cite{DL}, Thm. 1.6, it is shown that $X_2$ is universally homeomorphic to a line bundle over $X(sw')$ with the zero section removed. Let $g: X_2\to X(sw')$ be the corresponding morphism. Then the Leray spectral sequence gives a long exact sequence 
$$ \cdots \rightarrow  H^{i-1}_c(X(sw'), R^1g_!\mathcal F) \rightarrow  H^{i}_c(X_2, \mathcal F_{|X_2}) 
 \rightarrow  H^{i-2}_c(X(sw'), R^2g_!\mathcal F)\rightarrow  \cdots 
 $$
We have $\ell(sw')=\ell(w')+1.$ By induction $H^{i}_c(X(sw'), R^jg_!\mathcal F )=0$ for all
$i <\ell(sw')=\ell(w)-1$ and all $j$.   Thus $H^{i-1}_c(X(sw'), R^1g_!\mathcal F)=0$ and 
$H^{i-2}_c(X(sw'), R^2g_!\mathcal F)=0$ for $i<\ell(w).$ The claim follows. \qed

\section{Proof of Proposition \ref{PropositionII}}

We retain the notation of the statement of the proposition. It obviously suffices to prove the statement for a 
{\it maximal} parabolic subgroup $P$. Let $B$ be a Borel subgroup contained in $P$
and let $T$ be a maximal torus in $B.$ Let $M$ be the Levi subgroup
of $P$ containing $T.$ Then
\begin{equation*}
\begin{aligned}
\dim G/P  =& \  \dim G/B - \dim M/M\cap B \\
 =& \ | \Phi^+ | -  |\Phi^+_M| \ ,
\end{aligned}
\end{equation*}
where $\Phi^+= \Phi^+_G$ resp. $\Phi^+_M$  denotes the set of positive roots
of $G$ resp. of $M$. The assertion is now reduced to a purely combinatorial statement that can be checked mechanically for each type in 
the tables \cite{Bou}. We adopt the notation used in these tables.

\smallskip
\noindent ${\it Type} A_{\ell}$ : Here $|\Phi^+| = \frac{\ell(\ell+1)}{2}$. If $\Delta_M$ is obtained by deleting the root $\alpha_i$, then $\Phi_M$ is of type $A_{i-1} \times A_{\ell -i}$ (with the convention $A_0=\emptyset$). Hence $|\Phi^+_M| = \frac{i(i-1)}{2}+ \frac{(\ell-i)(\ell-i+1)}{2}. $
Hence $|\Phi^+| - |\Phi^+_M| \geq \ell$, with equality iff $i=1$ or $i=\ell.$

\smallskip
\noindent ${\it Type} B_\ell \,(\ell \geq 2)$ : Here $|\Phi^+| = \ell^{2}$. If $\Delta_M$ is obtained by deleting  $\alpha_i$, then $\Phi_M$ is of type $A_{i-1} \times B_{\ell -i}$ (with the convention $B_0=\emptyset$ , $B_1 =A_1$) and 
$|\Phi^+_M| = \frac{i(i-1)}{2} + (\ell-i)^{2}. $
Hence $|\Phi^+| - |\Phi^+_M| > \ell$ in all cases. The type $C_\ell$ is identical.

\smallskip
\noindent ${\it Type} D_\ell \,(\ell \geq 4)$ : Here $|\Phi^+| = (\ell -1) \ell$. If $\Delta_M$ is obtained by deleting  $\alpha_i$, then $\Phi_M$ is of type $A_{i-1} \times D_{\ell -i}$ 
except when $i=\ell-1$ or $i=\ell$ in which case $\Phi_M$ is of type $A_{\ell-1}$, and except when $i=\ell -2$
in which case $\Phi_M$ is of type $A_{\ell-3}\times A_1 \times A_1$, and except when $i=\ell-3$ in which case $\Phi_M$ is of type $A_{\ell-4} \times A_3.$ For $1\leq i \leq \ell -4$, 
$|\Phi^+_M| = \frac{i(i-1)}{2} + (\ell-i)(\ell-i-1)$ and hence $|\Phi^+| - |\Phi^+_M| > \ell.$ 
For $i=\ell-1$ or $i = \ell$, $|\Phi^+_M| = \frac{(\ell-1)\ell}{2}$ and for $i=\ell-2$, $|\Phi^+_M| = \frac{(\ell-3)(\ell-2)}{2} +2$, and for $i=\ell -3,$ $|\Phi^+_M| = \frac{(\ell -4) (\ell-3)}{2}+6.$
In all cases $|\Phi^+| - |\Phi^+_M| > \ell.$

\smallskip
For the exceptional types one gets for the differences $|\Phi^+| - |\Phi^+_M|$, as $\Delta_M$ is obtained by deleting $\alpha_1,\ldots,\alpha_\ell$,  the following integers:

\smallskip

$E_6  :  16, 21, 25, 29, 25, 16$

\smallskip 

$E_7  :  33, 42, 47, 53 , 50, 42, 27$ 

\smallskip

$E_8  :  78, 92, 98, 106, 104, 97, 83, 57$ 

\smallskip

$F_4  :  15,20, 20,15$

\smallskip

$G_2  :  5,5 .$

\smallskip
\noindent In each case the numbers are strictly larger than the rank. 
\qed

\section{Proof of Theorem \ref{Theorem}}
Let us first treat the case when $G_0$ is absolutely simple. Let us assume that $X=X_{G_0}(w)$ is universally homeomorphic to $X_{G_0}(\N)^{\rm ss}$. By Proposition \ref{PropositionI} we have $H^i_c(X)=0$ for $i<\ell(w)=\dim X.$ Comparing with Proposition \ref{Non-vanishing} we obtain
\begin{equation*}
 \dim X_{G_0}(w) \leq r_0 \ .
 \end{equation*}
Now the relative rank $r_0$ of $G_0$ is at most the absolute rank $r$. From Proposition \ref{PropositionII} we obtain the chain of inequalities
\begin{equation}\label{rankineq}
 \dim X_{G_0}(w) \leq r_0\leq r \leq  \dim X(\N)\ .
\end{equation}

Hence all inequalities are equalities and by  Proposition \ref{PropositionII},
we have that $G={\rm PGL}_n$ and that $\N$ corresponds to $(x,y^{(n-1)})$ or $(x^{(n-1)},y)$ with $x >y$.
Indeed, the case where $\N$ corresponds to $(x^{(n)})$ is excluded, since this would imply that $\ell(w)=\dim X(\N)=0$, hence $X(w)=X(\mF_q)$ would not be connected. Also the equality $r_0=r$ implies that $G_0={\rm PGL}_n$. It follows that $X_{G_0}(\N) \cong \Omega^n$ and
$\ell(w)=n-1.$ On the other hand, since $X(w)$ is connected, $w$ has to be an elliptic element in $S_n$, i.e.,
every fundamental reflection has to appear in a minimal expression of $w$, cf. \cite{L}, p.26, and \cite{BR} (the converse is also true, but more difficult to prove).
Hence every fundamental reflection
appears exactly once, i.e. $w$ is a Coxeter element. Now, the assertion follows from the remarks in the introduction.

Now let $G_0$ be of the form  $G_0={\rm Res}_{\mF_{q'}/ \mF_q}(G_0')$, where $G_0'$
is absolutely simple of adjoint type, and let  $t=|\mF_{q'}: \mF_q|$. As in the introduction we write $\N=(\N_1, \ldots , \N_t)$, where the $\N_i$ are conjugacy classes of $G'$. Let $r$ be the absolute rank of $G_0'$. Let $t_1$ be the number of indices $i$, where $\N_i$ is nontrivial. The inequality (\ref{rankineq}) is replaced by 
\begin{equation}\label{morerankineq}
 \dim X_{G_0}(w) \leq r_0\leq r t_1\leq  \dim X(\N)\ .
\end{equation}
Since $r_0\leq r$, we deduce from the fact all inequalities in (\ref{morerankineq}) are equalities,  that $r_0=r$ and $t_1=1$ (as before the case $t_1=0$ is excluded). As in the absolutely simple case we deduce that $G_0'={\rm PGL}_n$, and that for the  one index $i$ with non-trivial $\N_i$ this conjugacy class of ${\rm PGL}_n$ corresponds to 
 $(x,y^{(n-1)})$ or $(x^{(n-1)},y)$ with $x >y$. Reasoning as before, this implies that $w$ is a Coxeter element in the sense of \cite{L'}, i.e., every $F$-orbit  of simple reflections appears precisely once in a minimal expression of $w$ as a product of simple reflections.  All these Coxeter elements define universally homeomorphic DL-varieties, cf. \cite{L'}, Prop. 1.10. To identify the variety $X=X_{G_0}(w)=X_{G_0}(\N)^{\rm ss}$ with 
 $\Omega^n_{\mF_{q'}}$, one may use either incarnation  of $X$. On the DL-side, one can use the Coxeter element $w=(w_1, \ldots , w_t)$ with $w_1=s_1s_2\ldots s_{n-1}$ and  $w_2=\ldots=w_t=1$. Since the action of $F$ on the flag variety of $G_0$ is  given by $F(B_1, \ldots , B_t)=(F^tB_t, B_1, \ldots , B_{t-1})$,
 one sees easily that $X_{G_0}(w)\simeq \Omega^n_{\mF_{q'}}$. \qed

\section{The condition {\bf Aff}$(G_0,q,w)$}

In this section  we show that the condition  ${\bf Aff}(G_0,q,w)$ is  satisfied  for classical split groups.

We shall use the following criterion of Deligne and Lusztig \cite{DL}, 9.6. Let $C\subset X_\ast(T)_{\mR}$ be 
the (open) Weyl chamber. For $w \in W$, let 
\begin{equation*}
D(C,-w^{-1}C)=\{x\in  X_\ast(T)_{\mR}\mid \alpha(x)>0\ \forall \alpha >0 \ \text{ with } w(\alpha)<0 \}\ .
 \end{equation*}
 Here $\alpha$ ranges over the roots of $T$. 

\medskip
\noindent {\bf DL-Criterion:} {\it A DL-variety $X(w)$ is affine if there exists an element $x\in D(C,-w^{-1} C)$, such that $F^\ast x -w x \in C.$}

\begin{rk}\label{Haastertcomment}
It is not clear how close the Deligne-Lusztig criterion comes to being an equivalence. In \cite{H2}, Haastert checks
that for a split classical group, every conjugacy class of $W$ contains elements which satisfy the DL-criterion. However, there are not enough elements  of minimal length among his elements: there are  elements $w$ of minimal length in their conjugacy class
such that  there is no $w'$ among Haastert's elements  with $w\rightarrow_F w'$ (e.g. consider the root system $D_\ell$ and $w=t'$ below). Still, the method used below is modelled on Haastert's calculations. For $G_0$ of type $G_2$, he shows that the DL-criterion is satisfied for all $w\in W$ and all $q$, except $q=2$ and $w=s_1s_2s_1, s_2s_1s_2$, when it is not. We expect that  these last two DL-varieties are not affine. It should be possible to check this with the help of the computer.
\end{rk}

We now consider the root system of a split classical group. In \cite{GP}, Geck and Pfeiffer construct a subset of the Weyl group  which contains enough elements of minimal length in their conjugacy class. To recall their result,  we set up the notation as follows. Let $\mR^\ell=(\mR^\ell, (\,,\,))$ be the standard euclidian vector space with standard basis $\{e_1,\ldots,e_\ell\}$. We recall the sets of simple roots and simple reflections. The extraneous elements below are introduced 
to give a reasonably uniform treatment of all types.
 
\medskip

\noindent {\it Type} $A_{\ell-1}\, (\ell \geq 2)$ : $\Delta=\{\alpha_1,\ldots,\alpha_{\ell-1}\}$, where $\alpha_i=e_i-e_{i+1}$, $S=\{s_1,\ldots,  s_{\ell-1} \}$, where
$s_i=s_{\alpha_i}$. Further we set $s_i'=1$ for all $0\leq i \leq \ell -1$. The Weyl chamber is given by
$$C=\{(x_1,\ldots,x_\ell)\in \mR^\ell \mid  x_1>x_2 > \cdots >x_{\ell-1} > x_\ell, \sum\nolimits_i x_i=0 \}.$$
\medskip
\noindent {\it Type}  $B_\ell \, (\ell \geq 2)$ : $\Delta=\{\alpha_0=e_1,\alpha_1,\ldots,\alpha_{\ell-1}\}$, $S=\{t, s_1,\ldots, s_{\ell -1}\}$ where $t=s_{e_1}$, i.e,  $t(e_i)=e_i, \forall i\neq 1$ and $t(e_1)=-e_1.$
Further, we set $s_0'=t$ and $s_i'=s_i s_{i-1} \cdots s_1 t  s_1\cdots s_{i-1}s_i$ for  $1\leq i \leq \ell-1.$
The Weyl chamber is  given by
$$C=\{(x_1,\ldots,x_\ell)\in \mR^\ell \mid x_1 >0,\, x_1>x_2 > \cdots >x_{\ell-1} > x_\ell \}.$$
\medskip
\noindent {\it Type}  $D_\ell \, (\ell \geq 4)$ : $\Delta=\{\alpha_0=e_1+e_2,\alpha_1,\ldots,\alpha_{\ell-1}\}$, $S=\{t', s_1,\ldots, s_{\ell -1}\}$ where $t'=s_{e_1+e_2}$ is the reflection with $t'(e_1+e_2)=-(e_1 + e_2).$
Further we set $s_0'=t's_1$ and $s_i'=s_{i+1} s_{i} \cdots s_2 (t'\cdot s_1)  s_2\cdots s_{i}s_{i+1}$ for  $1\leq i \leq \ell-2.$ The Weyl chamber is given by
$$C=\{(x_1,\ldots,x_\ell)\in \mR^\ell \mid x_1 +x_2 >0,\, x_1>x_2 > \cdots >x_{\ell-1} > x_\ell \}.$$

\medskip
For a decomposition $\lambda=(\lambda_1,\ldots, \lambda_k)$ of the integer  $\ell$ in the cases $A_{\ell-1}$ and $B_{\ell}$, resp. of $\ell-1$ in the case
  $D_\ell$,  and a
 vector of signs $\epsilon=(\epsilon_1,\ldots, \epsilon_k)\in \{\pm 1  \}^{k}$,  let $w_{\lambda,\epsilon}=\prod_{i=1}^k w_{\lambda_i,\epsilon_i},$ where
$$w_{\lambda_i,\epsilon_i}:=\left\{\begin{array}{ccc} s_{m_i} s_{m_i+1} \cdots s_{n_i-1} & \mbox{ if } & \epsilon_i=1 \\   s_{m_i-1}'\cdot s_{m_i} s_{m_i+1} \cdots s_{n_i-1} & \mbox{ if } & \epsilon_i=-1 
                        \end{array} \right. $$
Here we put $m_i=\sum_{j=1}^{i-1} \lambda_j +1$ , $n_i = \sum_{j=1}^{i} \lambda_j,$ $i=1,\ldots, k$ in the cases $A_{\ell-1}$ and $B_{\ell}$, resp. $m_i=\sum_{j=1}^{i-1} \lambda_j +2$ , $n_i = \sum_{j=1}^{i}\lambda_j+1$ in the case  $D_\ell$. Since the elements $w_{\lambda_i,\epsilon_i}$ commute with each other, the above product makes sense. 

\begin{prop}\label{Weylelements} {\rm (\cite{GP} Prop. 2.3) }
For each $w\in W$, there is a decomposition $\lambda=(\lambda_1,\ldots, \lambda_k)$ of $\ell$ in the cases $A_{\ell-1}$ and $B_{\ell}$, resp. of $\ell-1$ in the case
  $D_\ell$,   and a vector of signs $\epsilon=(\epsilon_1,\ldots, \epsilon_k)\in \{\pm 1  \}^{k}$ such that $w \rightarrow_F \delta \cdot w_{\lambda,\epsilon},$
where  $\delta=1$ in the case of $A_{\ell-1}$ and $B_\ell$ and  $\delta \in \{1, s_1, t'\} $ in the case of $D_\ell.$
\end{prop}

\begin{rk}
In the case $A_{\ell-1}$ the elements $\delta \cdot w_{\lambda,\epsilon}$ are all minimal in
their conjugacy class; this is not true in the cases $B_\ell$ and $D_\ell$. In general, not all 
elements minimal in their conjugacy class are of the form $\delta \cdot w_{\lambda,\epsilon}$. 
\end{rk}

We note that to prove the condition ${\bf Aff}(G_0,q,w)$ for all $q, w$, it suffices to prove that for the elements of the form $\delta \cdot w_{\lambda,\epsilon}$ the corresponding DL-variety
is affine. Indeed, if $w$ is an element of minimal length in its conjugacy class, then by Proposition \ref{Weylelements} we find $w'=\delta \cdot w_{\lambda,\epsilon}$ with 
 $w \rightarrow_F w'$. Since then $\ell(w)=\ell(w')$, a repeated application of Lemma \ref{Lemma_gleicheLaenge}  shows that the DL-varieties $X(w)$ and $X(w')$ are universally homeomorphic. Hence
 the fact that $X(w')$ is affine implies that $X(w)$ is affine as well.
 
We will show that $X(w)$ is affine for elements 
$w=\delta \cdot w_{\lambda,\epsilon}$ by checking  the DL-criterion for $w$. In the split case $F$-conjugacy is simply conjugacy and  the action of the Frobenius $F^\ast$ is simply the multiplication by $q$. We will in fact even show that we can find $x\in C$ such that $qx-wx\in C$.

\medskip
\noindent{\it Type} $A_{\ell-1}:$ We will use the following lemma.

\begin{lemma}\label{LemmaA_l} Let $w=s_m\cdot s_{m+1} \cdots s_{n-1}$.
Let $x_1 > x_2 > \cdots > x_m > 0$ be positive real numbers. Then there exist $x_{m+1} > x_{m+2} > \cdots > x_{n+1}>0$  
with $x_m >x_{m+1}$, such that for any choice of $x_{n+2} > x_{n+3} >  \cdots > x_{\ell-1}>x_{\ell}>0$ with $x_{n+1} > x_{n+2},$
we have
$$(qx-wx,\alpha) > 0 \ \forall \alpha\in \Delta \ \text{and } (qx-wx,\alpha_n) > x_{n+1}\ .$$ 
\noindent {\rm [if $n=\ell$, the last condition is interpreted as empty, and the other chains of inequalities are to be interpreted in the obvious way.]}
\end{lemma}

\begin{proof}
We compute 
$$wx=(x_1,\cdots,x_{m-1},x_n,x_m,x_{m+1},\cdots, x_{n-1},x_{n+1},x_{n+2},\ldots,x_\ell)\ .$$
Thus we get for  $i\leq m-2$
\ and    for $i \geq n+1$,
$$(qx-wx,\alpha_i) =   q(x_i-x_{i+1})-(x_i -x_{i+1}) =(q-1)(x_i - x_{i+1})\ .$$ 
Furthermore, 
\begin{eqnarray}\label{identities}
\nonumber (qx-wx,\alpha_{m-1}) & = &  q(x_{m-1}-x_{m})-(x_{m-1} -x_{n}) \\
\nonumber (qx-wx,\alpha_{m}) & = &  q(x_{m}-x_{m+1})-(x_{n} -x_{m}) \\
(\nonumber qx-wx,\alpha_{m+1}) & = &  q(x_{m+1}-x_{m+2})-(x_{m} -x_{m+1}) \\
& \vdots & \\
\nonumber (qx-wx,\alpha_{n-1}) & = & q(x_{n-1}-x_{n})-(x_{n-2} -x_{n-1}) \\
\nonumber (qx-wx,\alpha_{n}) & = & q(x_{n}-x_{n+1})-(x_{n-1} -x_{n+1}) \ .\\ \nonumber 
\end{eqnarray}
We immediately see that $(qx - wx,\alpha_i)>0$ $\forall i\leq m-2$, $\forall i\geq n+1$ for any $x=(x_1>x_2>\cdots>x_\ell>0)\in \mR^\ell.$ 
For the remaining expressions, it suffices to treat the case $q=2.$
 For  $1\leq i \leq {n-m}$ set $x_{m+i}:=x_m-ia$ with $a>0$.  Then
\begin{eqnarray*} 
(2x-wx,\alpha_{m-1}) & = &  2(x_{m-1}-x_{m})-(x_{m-1} -x_{n}) = (x_{m-1}-x_{m})-(x_{m} -x_{n}) \\
& = &  (x_{m-1}-x_{m})-(n-m)a.
\end{eqnarray*}
This expression is positive if $a$ is small enough. 
The inequality  $$(2x-wx,\alpha_{m}) = 2(x_{m}-x_{m+1})-(x_{n} -x_{m}) >0 $$ is clearly satisfied since $x_m > x_n$.
Since $(2x-wx,\alpha_i)= a>0$ for $m+1 \leq i \leq n-1,$
it remains to consider the term  $(2x-wx,\alpha_{n}).$ But $(2x-wx,\alpha_{n}) = 2(x_{n}-x_{n+1})-(x_{n-1} -x_{n+1}) > x_{n+1}>0$
if $$2(x_{n}-x_{n+1})-x_{n-1}  > 0.$$
Set $x_{n+1}=x_n-b$ with $0 < b< x_n.$ Then
$2(x_n-x_{n+1})> x_{n-1}$,  provided that $b> \frac{x_{n-1}}{2}.$ If $a$ is small enough, such that $2x_n > x_{n-1}$, such  $b>0$ exists.
\end{proof}

\begin{prop}\label{ersteProp}
Let $w=w_{\lambda,\epsilon} \in W$. Then there is an $x=(x_1 > x_2 > \cdots > x_\ell>0)\in  \mR^\ell$ with  $(qx-wx,\alpha)>0$ for all $\alpha \in \Delta.$
\end{prop}

\begin{proof}
Let $w=w_{\lambda_1,\epsilon_1}\cdots w_{\lambda_k,\epsilon_k}$ and put $w_i=w_{\lambda_i,\epsilon_i}$. Note that the vector of signs $\epsilon$  does not affect
this element.
Set $x_1=1$ and apply  successively Lemma \ref{LemmaA_l} (starting with $w_1$).  
We have $(qx-wx,\alpha_k)= (qx-w_ix,\alpha_k)>0$ for $k\in [m_i,n_i-1].$ Further,
\begin{eqnarray*}
(qx-wx,\alpha_{n_i}) - (qx-w_ix,\alpha_{n_i}) & = & -(x_{w^{-1}(n_i)}- x_{w^{-1}(n_i+1)}  ) + (x_{w_i^{-1}(n_i)} - x_{n_i+1}) \\  & = & -x_{n_i+1} + x_{w^{-1}(n_i+1)}.
\end{eqnarray*}
Thus $(qx-wx,\alpha_{n_i}) > 0$,  since we arranged in Lemma \ref{LemmaA_l} that $(qx-w_ix,\alpha_{n_i})> x_{n_i+1}.$

\end{proof}
\begin{coro}
There exists $x\in C$ with  $(qx-wx,\alpha)>0$ for all $\alpha \in \Delta.$
\end{coro}
\begin{proof}
We add to the $x$ in Proposition \ref{ersteProp} a multiple $r\cdot(1,1,\ldots,1)$ such that $x+r\cdot(1,1,\ldots,1)\in C.$
\end{proof}

\medskip

\noindent{\it Type} $B_\ell:$ In this case we use the following lemma.

\begin{lemma}\label{LemmaB_l} Let $w=s_m\cdots s_{n-1}$ or $w=s_{m-1}'s_m\cdots s_{n-1}.$
Let $x_1 > x_2 > \cdots > x_{m-1} > x_m > 0$ be positive real numbers with $x_{m-1}> 3x_m$ if $m\geq 2.$ Then there exist 
$x_{m+1} > x_{m+2} > \cdots > x_{n+1}>0$ with $x_m >x_{m+1}$ and $x_n > 3 x_{n+1}$, such that for any choice of 
$x_{n+2}>x_{n+3} > \cdots >x_\ell>0$ with $x_{n+1} > x_{n+2},$ 
we have
$$(qx-wx,\alpha) > 0 \ \forall \alpha \in \Delta \text{ and } (qx-wx,\alpha_n) > 2x_{n+1}\ .$$ 
\noindent {\rm [if $n=\ell$, the last condition is interpreted as empty.]}
\end{lemma}

\begin{proof} The case of $w=s_m\cdots s_{n-1}$  is similar to the one treated in Lemma \ref{LemmaA_l}.   
We only have to check in addition  that  $(qx-wx,e_1)>0$ which is clear.
      
So, let $w=s_{m-1}'s_m\cdots s_{n-1}.$
Again, it suffices to consider the case $q=2.$ We compute
$$wx=(x_1,\cdots,x_{m-1},-x_n,x_m,x_{m+1},\cdots, x_{n-1},x_{n+1},x_{n+2},\ldots,x_\ell).$$
We get the same system of identities (\ref{identities}) as in the proof of Lemma \ref{LemmaA_l} except for the first two, which now become 
\begin{equation}\label{first}
(2x-wx,\alpha_{m-1}) = 2(x_{m-1}-x_{m})-(x_{m-1} + x_{n}) 
\end{equation}
and
$$(2x-wx,\alpha_{m}) = 2(x_{m}-x_{m+1})+(x_{n}+x_{m}). $$
We also have to check that $(2x-wx,e_1)>0.$ This is easy since 
$$(2x-wx,e_1)=\left\{ \begin{array}{ccc} x_1 & \mbox{ if } & m \neq  1 \\ 2x_1+x_{n} & \mbox{ if } & m = 1 \end{array} \right. .$$
We only have to care of the first expression (\ref{first}). We set  $x_{m+i}:=x_m-ia$ with $a>0$ for $1\leq i \leq {n-m}$ and write 
\begin{eqnarray*}
(2x-wx,\alpha_{m-1}) & = & 2(x_{m-1}-x_{m})-(x_{m-1}+ x_{n}) = (x_{m-1}-x_{m})-(x_{m} + x_{n}) \\ & = & (x_{m-1}- 3x_{m})+(n-m)a .
\end{eqnarray*}
Since we have $x_{m-1} > 3 x_{m}$,  this term is positive. 
Finally, we have to show that
$$(2x-wx,\alpha_{n}) = 2(x_{n}-x_{n+1})-(x_{n-1} -x_{n+1}) > 2x_{n+1}$$
and 
$$ x_n > 3x_{n+1}.$$
Write $x_{n+1}=x_n-b$ with $0<b<x_n.$ Then the first of the above inequalities becomes
$2b > x_{n-1} + x_{n+1}$, i.e., 
$$3b > x_n+ x_{n-1} $$
and the second becomes 
$$ 3b > 2x_n.$$
Similarly as in Lemma \ref{LemmaA_l}, we can find  $b$ such that these inequalities are solvable.
\end{proof}

\begin{prop}\label{PropositionB_l}
Let $w=w_{\lambda,\epsilon} \in W$. Then there is an $x=(x_1 > x_2 > \cdots > x_\ell >0)\in \mR^\ell$ with $(qx-wx,\alpha)>0$ 
for all $\alpha \in \Delta.$
\end{prop}

\begin{proof}
The proof is the same as in the case of $A_{\ell-1}$, except that we have
$$(qx-wx,\alpha_{n_i}) - (qx-w_ix,\alpha_{n_i}) =-x_{n_i+1} \pm x_{w^{-1}(n_i+1)}.$$
Thus $(qx-wx,\alpha_{n_i}) > 0$  since we made sure in Lemma \ref{LemmaB_l} that  $(qx-w_ix,\alpha_{n_i})> 2x_{n_i+1}.$
\end{proof}
Note that the $x$ in Proposition \ref{PropositionB_l} lies in $C$. 

\medskip

\noindent{\it Type} $D_\ell:$ In this case we prove the following proposition.

\begin{prop}\label{PropositionD_l}
Let $w=\delta w_{\lambda,\epsilon} \in W$. Then there is an $x=(x_1 > x_2 > \cdots > x_\ell >0) \in \mR^\ell$ with $(qx-wx,\alpha)>0$ for all 
$\alpha \in \Delta.$
\end{prop}

\begin{proof}
The element $s_0'=s_1t'$ is the reflection with $e_1 \mapsto -e_1$, $e_2 \mapsto -e_2$ and which fixes all other $e_j$. It follows that
$s_i'(e_1)=-e_1$, $s_i'(e_{i+2})=-e_{i+2}$ and $s_i'(e_j)=e_j$ for all $j\not\in \{1,i+2\}.$

Let $\lambda=(\lambda_1,\ldots, \lambda_k)$ be a decomposition of the closed interval $[2,\ell]$ and consider  a vector of signs
$\epsilon=(\epsilon_1,\ldots,\epsilon_k)$. Set $|\epsilon|:= \#\{ i\mid \epsilon_i < 0\}.$ Then one computes that the element
$w=w_{\lambda,\epsilon}$ is given by
$$e_1 \mapsto (-1)^{|\epsilon|} e_1,$$
$$ e_{m_1} \mapsto e_{m_1+1}, e_{m_1+1} \mapsto e_{m_1+2}, \ldots, e_{n_1-1} \mapsto e_{n_1}, e_{n_1} \mapsto (-1)^{\epsilon_1} e_{m_1} $$  
$$e_{m_2} \mapsto e_{m_2+1}, e_{m_2+1} \mapsto e_{m_2+2},\ldots, e_{n_2-1} \mapsto e_{n_2}, e_{n_2} \mapsto (-1)^{\epsilon_2} e_{m_2} $$ 
\centerline{\vdots}
$$e_{m_k} \mapsto e_{m_k+1}, e_{m_k+1} \mapsto e_{m_k+2},\ldots, e_{n_k-1} \mapsto e_{n_k}, e_{n_k} \mapsto (-1)^{\epsilon_k} e_{m_k} \ .$$   
It follows  that $w_{\lambda,\epsilon}$ corresponds to the element $w_{\tilde{\lambda},\tilde{\epsilon}}$ of $W(B_\ell)$ with
$\tilde{\lambda}=(1,\lambda)$ and $\tilde{\epsilon}=((-1)^{|\epsilon|},\epsilon).$

If we multiply $w_{\lambda,\epsilon}$ by $\delta\in \{1,s_1,t' \}$ from the left, only the first factor $w_{\lambda_1,\epsilon_1}$ of $w_{\lambda,\epsilon}$ is affected.
In particular, we may reduce by Lemma \ref{LemmaB_l} to the case $\lambda=(\lambda_1,1,\ldots,1).$

\noindent {\bf Case}: $\delta=1.$ 

 Let $x\in \mR^\ell$ be  chosen as in Proposition \ref{PropositionB_l}. Then $(qx-wx,\alpha_i)> 0$ for all $i\geq 1.$
So, we only have to ensure that $(qx-wx,\alpha_0)=q(x_1+x_2)- ((-1)^{|\epsilon|}x_1\pm x_{w^{-1}(2)}) > 0$ which is clearly satisfied.

\medskip

\noindent {\bf Case}: $\delta=s_1.$

\noindent {\bf Subcase}: $m_1=2.$ 
Then $w=\delta\cdot w_{\lambda_1,\epsilon_1}$ is given by
$$e_1 \mapsto (-1)^{|\epsilon|} e_2,$$
$$ e_{2} \mapsto e_{3}, e_{3} \mapsto e_{4}, \ldots, e_{n_1-1} \mapsto e_{n_1}, e_{n_1} \mapsto (-1)^{\epsilon_1} e_{1} \ .$$  
If $|\epsilon|$ is even, this case is treated in Lemma \ref{LemmaB_l}.
So, let $|\epsilon|$ be odd. We compute
$$wx=((-1)^{\epsilon_1}x_{n_1},-x_1,x_2,x_3,\ldots,x_{n_1-1},x_{n_1+1},\ldots,x_\ell) \ .$$
Hence
\begin{eqnarray*}
(qx-wx,\alpha_0) & = &  q(x_1+x_{2})+ (x_1 + (-1)^{\epsilon_1+1}x_{n_1}) > 0 \\ 
(qx-wx,\alpha_1) & = &  q(x_1-x_{2}) - x_1 - (-1)^{\epsilon_1}x_{n_1} \\ 
(qx-wx,\alpha_2) & = &  q(x_2-x_{3}) + (x_1 + x_2) > 0 \\
(qx-wx,\alpha_3) & = &  q(x_3-x_{4}) - (x_2 - x_3) \\
& \vdots & \\
(qx-wx,\alpha_{n_1-1}) & = &  q(x_{n_1-1}-x_{n_1}) - (x_{n_1-2} - x_{n_1-1}) \\
(qx-wx,\alpha_{n_1}) & = &  q(x_{n_1}-x_{n_1+1}) - (x_{n_1-1} - x_{n_1+1}) \\
\end{eqnarray*}
If $\epsilon_1$ is even then we choose $x_1>x_2 > \cdots > x_{n}>0$ in the following way. Let $x_2>0$ be arbitrary and set as in Lemma \ref{LemmaB_l}
 $x_{2+i}=x_2-ia, \, i=0,\ldots,n-3,$ with $a>0$ small enough. Further, let $x_n >0$ be such that $x_{n-1}> 3x_n.$ Finally choose $x_1>x_2$
 such that $x_1-x_2 > x_2-x_n.$ One checks that the above expressions are  positive.

If $\epsilon_1$ is odd then one chooses $x$ similarly.

\smallskip
\noindent {\bf Subcase: $m_1 > 2.$}

In this case, one reduces by Lemma \ref{LemmaB_l} to the situation of $w=s_1$ resp. $w=s_1\cdot t.$

\medskip
\noindent {\bf Case}: $\delta=t'.$

\noindent {\bf Subcase: $m_1 = 2$.} Then $w=\delta w_{\lambda_1,\epsilon_1}$ is given by
$$e_1 \mapsto (-1)^{|\epsilon|+1} e_2,$$
$$ e_{2} \mapsto e_{3}, e_{3} \mapsto e_{4}, \ldots, e_{n_1-1} \mapsto e_{n_1}, e_{n_1} \mapsto (-1)^{\epsilon_1+1} e_{1} \ .$$  
These cases  are already covered by the previous one.

\smallskip
\noindent {\bf Subcase: $m_1 > 2.$}

In this case, one reduces by Lemma \ref{LemmaB_l} to the situation of $w=t'$ resp. $w=t'\cdot t.$  

\end{proof}
Note that the $x$ in Proposition \ref{PropositionD_l} lies in $C$.

\end{document}